\documentclass[12pt,final]{amsart} 
 
\usepackage{fullpage}
\usepackage[notcite,notref,color]{showkeys}

\usepackage{amsmath,amssymb,amsbsy,amsthm,bbm} 
\usepackage{enumitem} 
\usepackage[overload]{textcase}
 
\usepackage{epsfig,psfrag,color,epstopdf} 
\usepackage{verbatim,xspace}

\usepackage{tikz}

\usepackage[colorlinks,bookmarks,linkcolor=black,citecolor=black]{hyperref}

\colorlet{vertexcolor}{black}
\colorlet{edgecolor}{black}

\newlength{\edgewidth}
\tikzstyle{vertex}=[circle, draw, fill=vertexcolor, inner sep=0pt, minimum size=4pt]
\tikzstyle{edge}=[edgecolor,line width=\edgewidth]

\newcommand{\edge}[2] {
 \draw[edge] (#1) -- (#2);
}

\newcommand{\vertex}[4][] {
  \node (#2) at (#3) {};
  \node at (#2) [vertex, label=#4:#1] {};
}

\newcommand{\nvertex}[3][] { \vertex[#1]{#2}{#3}{above}} 
 
\newcommand{\evertex}[3][] { \vertex[#1]{#2}{#3}{right}} 
\newcommand{\wvertex}[3][] { \vertex[#1]{#2}{#3}{left}}

\setlist{itemsep=3pt,parsep=0pt,topsep=2pt,partopsep=0pt}  
\setlist{leftmargin=2.5\parindent} 

 \def\endofClaim{\hfill\scalebox{.6}{$\Box$}}

\let\subset\subseteq \let\eps\varepsilon \let\rho\varrho
\def\dcup{\cup}

\def\le{\leqslant} \def\ge{\geqslant}

\newtheorem{theorem}{Theorem}[section]
\newtheorem*{SzRLdeg}{Szemer\'edi's Regularity Lemma (sparse minimum degree form)}

\newtheorem{lemma}[theorem] {Lemma}    
    
\newtheorem{prop}[theorem] {Proposition}

\newtheorem{conj}[theorem]{Conjecture}
\newtheorem{prob}[theorem]{Problem} 
\newtheorem{qu}[theorem]{Question}  
\theoremstyle{definition}

\newtheorem{defn}[theorem]{Definition}

\theoremstyle{remark}

\newcommand{\oldqed}{}

\newenvironment{claimproof}[1][Proof]{
  \renewcommand{\oldqed}{\qedsymbol}
  \renewcommand{\qedsymbol}{\endofClaim}
  \begin{proof}[#1]
}{
  \end{proof}
  \renewcommand{\qedsymbol}{\oldqed}
}
\newcommand{\By}[2]{\overset{\mbox{\tiny{#1}}}{#2}} 
\newcommand{\ByRef}[2]{   \By{\eqref{#1}}{#2} }

\newcommand{\eqByRef}[1]{ \ByRef{#1}{=} }

 \def\Ex{\mathbb{E}}  \def\Pr{\mathbb{P}}
   \def\N{\mathbb{N}}

\newcommand{\Bin}{\textup{Bin}}

\def\Var{\textup{Var}}

\title{Chromatic thresholds in sparse random graphs}

  \author[P. Allen, J. B\"ottcher, S. Griffiths, Y. Kohayakawa \and R. Morris]{Peter Allen, Julia B\"ottcher, Simon Griffiths \\
 Yoshiharu Kohayakawa \and Robert Morris}
 \address{
    Peter Allen, Julia B\"ottcher \hfill\break
    Department of Mathematics, London School of Economics, Houghton Street, London, WC2A~2AE, UK.
 } 
 \email{p.d.allen|j.boettcher@lse.ac.uk}
\address{
    Simon Griffiths \hfill\break
    Department of Statistics, University of Oxford, 1 South Parks Road,
    Oxford, OX1~3TG, UK.
} 
\email{griffith@stats.ox.ac.uk}
\address{
    Yoshiharu Kohayakawa \hfill\break Instituto de Matem\'atica e
    Estat\'{\i}stica, Universidade de S\~ao Paulo, Rua do Mat\~ao 1010,
    05508--090~S\~ao Paulo, Brazil.
  }
  \email{yoshi@ime.usp.br}
  \address{
    Robert Morris\hfill\break
   IMPA, Estrada Dona Castorina 110, Jardim Bot\^anico, Rio de Janeiro, RJ,
   Brazil
 }
 \email{rob@impa.br}

\thanks{
    PA was partially supported by FAPESP (Proc.~2010/09555-7); JB by FAPESP
    (Proc.~2009/17831-7); SG by CNPq (Proc.~500016/2010-2); YK by CNPq
    (Proc.~308509/2007-2); RM by CNPq (Proc.~479032/2012-2 and Proc.~303275/2013-8). This
    research was supported by CNPq (Proc.~484154/2010-9). The authors are
    grateful to NUMEC/USP, N\'ucleo de Modelagem Estoc\'astica e Complexidade of
    the University of S\~{a}o Paulo, and Project MaCLinC/USP, for supporting
    this research. }

\date{\today}

\begin{document}

\begin{abstract}
  The chromatic threshold $\delta_\chi(H,p)$ of a graph $H$ with respect to
  the random graph $G(n,p)$ is the infimum over $d > 0$ such that the
  following holds with high probability: the family of $H$-free graphs $G \subset G(n,p)$ with minimum degree $\delta(G) \ge dpn$ has bounded chromatic number. The study of $\delta_\chi(H) := \delta_\chi(H,1)$ was initiated
  in 1973 by Erd\H{o}s and Simonovits. Recently $\delta_\chi(H)$ was
  determined for all graphs $H$. It is known that $\delta_\chi(H,p) =
  \delta_\chi(H)$ for all fixed $p \in (0,1)$, but that typically
  $\delta_\chi(H,p) \ne \delta_\chi(H)$ if $p = o(1)$. 

  Here we study the problem for sparse random graphs. We determine $\delta_\chi(H,p)$ for most functions~$p = p(n)$ when $H\in\{K_3,C_5\}$, and also for all graphs $H$ with $\chi(H) \not\in \{3,4\}$.
    \end{abstract}

\maketitle

\section{Introduction}

An important recent trend in combinatorics has been the formulation and proof of so-called `sparse random analogues' of many classical extremal and structural results. For example, Kohayakawa, \L uczak and R\"odl~\cite{KLRroth} conjectured almost twenty years ago that Szemer\'edi's theorem on $k$-term arithmetic progressions should hold in a $p$-random subset of $\{1,\ldots,n\}$ if $p \gg n^{-1/(k-1)}$, and that the Erd\H{o}s-Stone theorem should hold in the Erd\H{o}s-R\'enyi random graph $G(n,p)$ if $p \gg n^{-1/m_2(H)}$, where $m_2(H)$ is defined to be the maximum of $\frac{e(F) - 1}{v(F) - 2}$ over all subgraphs $F \subset H$ with $v(F) \ge 3$. The study of these conjectures recently culminated in the extraordinary breakthroughs of Conlon and Gowers~\cite{ConGow} and Schacht~\cite{SchTuran}, who resolved these conjectures (and many others), and in the development of the so-called `hypergraph container method', see~\cite{BMS,ST}. 

In this paper we study a sparse random analogue of the \emph{chromatic
  threshold} $\delta_\chi(H)$ of a graph $H$, which is defined to be the
infimum over $d > 0$ such that every $H$-free graph on $n$ vertices with
minimum degree at least $d n$ has bounded chromatic number. The study of chromatic thresholds was initiated in 1973 by Erd\H{o}s and Simonovits~\cite{ES73}, who were motivated by Erd\H{o}s' famous probabilistic proof~\cite{Erd59} that there exist graphs with arbitrarily high girth and chromatic number. Building on work of (amongst others) \L uczak and Thomass\'e~\cite{LT} and Lyle~\cite{Lyle10}, the chromatic threshold of every graph $H$ was finally determined in~\cite{ChromThresh}, where it was proved that
\begin{equation}\label{eq:CT:thm}
\delta_\chi(H) \, \in \, \bigg\{ \frac{r-3}{r-2}, \, \frac{2r-5}{2r-3}, \, \frac{r-2}{r-1} \bigg\}
\end{equation}
for every graph $H$ with chromatic number $\chi(H) = r \ge 2$. For example, $\delta_\chi(K_3) = 1/3$ and $\delta_\chi(C_{2k+1}) = 0$ for every $k \ge 2$, as was first proved by Thomassen~\cite{Thomassen02,Thomassen07}. 

We will study the following analogue of $\delta_\chi(H)$ in $G(n,p)$. 

\begin{defn}\label{def:deltachip}
Given a graph $H$ and a function $p = p(n) \in [0,1]$, define
\begin{align*}
& \delta_\chi\big( H, p \big) \, := \, \inf \Big\{ d > 0 \,:\, \text{there exists $C > 0$ such that the following holds} \\ 
& \hspace{5cm} \text{with high probability: every $H$-free spanning subgraph }  \\
& \hspace{6.5cm} G \subset G(n,p) \textup{ with } \delta(G) \ge d pn \textup{ satisfies } \chi(G) \le C \Big\}.
\end{align*}
We call $\delta_\chi\big(H,p\big)$ the \emph{chromatic threshold of $H$ with respect to $p$}.
\end{defn}

Note that $\delta_\chi(H) = \delta_\chi(H,1)$, so this definition generalises that of Erd\H{o}s and Simonovits. We emphasise that the constant $C$ is allowed to depend on the graph $H$, the function $p$ and the number $d$, but not on the integer $n$. We also note that if, for some $d$, with high probability there is no spanning $H$-free subgraph of $G(n,p)$ whose average degree exceeds $dpn$, then vacuously we have $\delta_\chi(H,p)\le d$.

In a companion paper~\cite{dense} we showed that $\delta_\chi( H, p ) = \delta_\chi(H)$ for all fixed $p > 0$, and made some progress on determining $\delta_\chi( H, p )$ in the case $p = n^{-o(1)}$. Here we begin the investigation of $\delta_\chi( H, p )$ for sparser random graphs. In particular, we will prove the following theorem, which determines $\delta_\chi( H, p )$ precisely for a large class of graphs and essentially all functions $p = p(n)$. We will write $\pi(H)$ for the Tur\'an density $1 - \frac{1}{\chi(H) - 1}$ of a graph $H$. 

\begin{theorem}\label{thm:classhigh} 
Let $H$ be a graph with $\chi(H) \not\in \{3,4\}$. Then
\begin{equation}\label{eq:thm:classhigh}
\delta_{\chi}(H,p) = \left\{
\begin{array}{cll}
\delta_\chi(H) & \text{if } & p > 0 \text{ is constant,}\smallskip\\ 
\pi(H) & \text{if } & n^{-1/m_2(H)} \ll p \ll 1,\smallskip\\
1 & \text{if } & \frac{\log n}{n} \ll p \ll n^{-1/m_2(H)}.
\end{array} \right.
\end{equation}
The same moreover holds for all graphs $H$ with $\chi(H) = 4$ and $m_2(H) \ge 2$.
\end{theorem}

We remark that~\eqref{eq:thm:classhigh} does \emph{not} hold for all $3$-chromatic graphs; for example, it does not hold when~$H= C_{2k+1}$ for any $k \ge 2$, see Theorems~\ref{thm:C5} and~\ref{thm:Clong} below. We suspect that it does not hold for every $4$-chromatic graph, though we do not have a counter-example, see Proposition~\ref{prop:inf42} and Conjecture~\ref{conj:chi4}. Note that if $p \ll \frac{\log n}{n}$ then $\delta_{\chi}(H,p) = 0$, since with high probability $G(n,p)$ has an isolated vertex, so the condition in the definition holds vacuously. We also remark that the proof of Theorem~\ref{thm:classhigh} uses a general probabilistic lemma (Lemma~\ref{lem:EF}), which appears to be new, and may be of independent interest.   

For graphs $H$ with $\chi(H) = 3$, it was shown in~\cite{dense} that the situation is significantly more complicated, even in the case $p = n^{-o(1)}$. We will therefore restrict ourselves to studying one specific family of particular interest, namely the odd cycles. (For a brief discussion of more general 3-chromatic graphs, see Section~\ref{sec:open}.) For $K_3$ we will show that the pattern is the same as in Theorem~\ref{thm:classhigh}.

\begin{theorem}\label{thm:K3} 
We have
\begin{equation*}
\delta_{\chi}(K_3,p) = \left\{
\begin{array}{cll}
\tfrac{1}{3} & \text{if } & p > 0 \text{ is constant}\smallskip\\ 
\tfrac{1}{2} & \text{if } & n^{-1/2} \ll p \ll 1\smallskip\\
1 & \text{if } & \frac{\log n}{n} \ll p \ll n^{-1/2}.
\end{array} \right.
\end{equation*}
\end{theorem}

For $C_5$, on the other hand, the behaviour is more complex. 

\begin{theorem}\label{thm:C5} 
We have
\begin{equation*}
\delta_{\chi}(C_5,p) = \left\{
\begin{array}{cll}
0 & \text{if } & p > 0 \text{ is constant}\smallskip\\ 
\tfrac{1}{3} & \text{if } & n^{-1/2} \ll p \ll 1\smallskip\\
\tfrac{1}{2} & \text{if } & n^{-3/4} \ll p \ll n^{-1/2}\smallskip\\
1 & \text{if } & \frac{\log n}{n} \ll p \ll n^{-3/4}.
\end{array} \right.
\end{equation*}
\end{theorem}

We believe that this is the first example of a random analogue of a natural extremal result which has been shown to exhibit several non-trivial phase transitions. Finally, for longer odd cycles we will obtain only the following partial description.

\begin{theorem}\label{thm:Clong} 
For every $k \ge 3$, we have
\begin{equation*}
\delta_{\chi}(C_{2k+1},p) = \left\{
\begin{array}{cll}
0 & \text{if } & p \gg n^{-1/2}\smallskip\\
\frac{1}{2} & \text{if } & n^{-(2k-1)/2k} \ll p \ll n^{-(2k-3)/(2k-2)}\smallskip \\
1 & \text{if } & \frac{\log n}{n} \ll p \ll n^{-(2k-1)/2k}.\smallskip
\end{array} \right.
\end{equation*}
\end{theorem}

We suspect that $\delta_{\chi}(C_{2k+1},p) = 0$ for all $p \gg
n^{-(k-2)/(k-1)}$, see Section~\ref{sec:open}, and that the techniques
introduced in this paper could potentially be extended to prove this
conjecture. However, we do not know what value to expect for
$\delta_{\chi}(C_{2k+1},p)$ in the remaining range.

Finally, we remark that Theorem~\ref{thm:Clong} proves a particular case of a much more general conjecture (see~\cite[Conjecture~1.6]{dense}) which states that, in the range $p = n^{-o(1)}$, the 3-chromatic graphs $H$ for which $\delta_\chi(H,p) = 0$ are exactly the `thundercloud-forest graphs', see Definition~\ref{def:cloudforest}. Together with the results of this paper and those in~\cite{dense}, this conjecture (if true) would completely characterise $\delta_\chi(H,p)$ in this range. 

\subsection*{Organisation of the paper}

We begin, in Section~\ref{sec:rob}, by establishing that $G(n,p)$ contains a
small subgraph with high girth, chromatic number and expansion (see
Proposition~\ref{prop:rob}). This result will be used in the lower bound
constructions for Theorems~\ref{thm:classhigh}, \ref{thm:K3}, and
\ref{thm:C5}.  In Section~\ref{sec:classhigh} we prove
Theorem~\ref{thm:classhigh} and construct an infinite family of $4$-chromatic
graphs~$H$ with $m_2(H)<2$ (see Proposition~\ref{prop:inf42}). In
Section~\ref{sec:lower} we provide the lower bound constructions for
Theorems~\ref{thm:K3} and~\ref{thm:C5}. We remark that some of these
constructions are given in a more general framework and reused
in~\cite{dense}. In Section~\ref{sec:upper} we prove the remaining upper
bounds that imply Theorems~\ref{thm:K3}, \ref{thm:C5},
and~\ref{thm:Clong}. We conclude with some open problems in
Section~\ref{sec:open}.

We remark that we often omit floors and ceilings whenever this does not
affect the argument.

\section{Small subgraphs of $G(n,p)$ with large girth and good expansion}\label{sec:rob}

In this section we will prove the following proposition, which will be a key tool in the proof of Theorem~\ref{thm:classhigh}, and in the lower bound constructions in Theorems~\ref{thm:K3} and~\ref{thm:C5}.

\begin{prop}\label{prop:rob}
For every $k \in \N$ and $\eps > 0$, and every function $n^{-1/2} \ll p = o(1)$, the following holds with high probability: there exists a subgraph $G \subset G(n,p)$ such that
\[v(G) \le \eps / p\,, \qquad \chi(G) \ge k \qquad \text{and} \qquad \textup{girth}(G) \ge k\,,\]
and moreover, for every pair $A,B \subset V(G)$ of disjoint vertex sets of size at least $\eps |G|$, the graph $G[A,B]$ contains an edge.
\end{prop}

Observe that this result would be easy if $v(G)\le\eps/p$ were replaced with $v(G)\le K/p$ for some large constant $K$, because $G(K/p,p)$ with high probability has the desired properties. So the difficulty is to obtain a much smaller subgraph with these properties. Similarly, a random graph on $m:=\eps/p$ vertices with $k^2m$ edges has a subgraph with the desired properties (see Lemma~\ref{lemma:Erdos}). This is denser than our $G(n,p)$ but
we will show that with high probability some set of $m$ vertices contains exactly $k^2 m$ edges (see Lemma~\ref{lem:var}).

However, this is not enough to conclude that $G(n,p)$ contains a subgraph with the desired properties. Indeed, writing $E_A$ for the event that $A$ contains exactly $k^2 |A|$ edges, and $F$ for the event that the desired subgraph $G$ exists, we just argued that 
$$\Pr\bigg( \bigcup_{|A| = m} E_A \bigg) = 1 - o(1) \qquad \text{and} \qquad \Pr\big( F \,|\, E_A \big) = 1 - o(1) \quad \text{for every $|A| = m$.}$$
Now suppose that each $E_A$ were the disjoint union of $F$ and $E'_A$, where the $E'_A$ are disjoint events, with $\Pr(E'_A) \ll \Pr(F)$ for each $i$, but $\Pr(F) \ll \sum_{|A| = m} \Pr(E'_A)$. In other words, each event $F|E_A$ is very likely for the same reason. If this were the case then $F$ would be a very unlikely event.

In this scenario, however, the random variable $X$ counting the number of $E_A$ which occur is not concentrated near its expectation, that is, $\Var(X)$ is not small compared to $\Ex[X]^2$. We will show in Lemma~\ref{lem:var} that in reality that is not the case. 
This combined with the following lemma suffices to prove Proposition~\ref{prop:rob}. 
We formulate this lemma for a general setup and believe it is likely to have other applications.

\begin{lemma}\label{lem:EF}
Let $E_1,...,E_t$ and $F$ be events with non-zero probability, and let $X = \sum_{j=1}^t \mathbbm{1}{[E_j]}$. Suppose that for some $\delta\ge0$ we have $\Var(X) \le \delta \Ex[X]^2$ and $\Pr\big( F \,|\, E_j \big) \ge 1 - \delta$ for each $j\in[t]$. Then $\Pr(F) \ge 1-6\delta$.
\end{lemma}

\begin{proof}
By Chebyshev's inequality, we have $\Pr\big(X\le\tfrac12 \Ex[X]\big)\le4\delta$. Therefore, if we define
$$Q \, := \, F^c \cap \big\{ X >\tfrac12 \Ex[X] \big\},$$
it will suffice to prove that $\Pr(Q) \le 2\delta$. 

So suppose that to the contrary we have $\Pr(Q) > 2\delta$. Observe that
\[\sum_{j=1}^t \Pr( E_j \cap Q ) \,=\, \sum_{\omega\in Q} \Pr(\omega) \sum_{j=1}^t \mathbbm{1}{[\omega \in E_j]} \, > \,   \Pr(Q)\cdot\tfrac12 \Ex[X]  \, > \, \delta \cdot \sum_{j=1}^t \Pr( E_j )\,,\]
where the first inequality follows from the definition of $Q$. By the pigeonhole principle, it follows that $\Pr( E_j \cap Q ) > \delta \cdot \Pr(E_j)$ for some $j \in [t]$. But then
\[\delta \,<\, \frac{\Pr( E_j \cap Q )}{ \Pr(E_j)} \,=\, \Pr\big( Q \,|\, E_j \big) \, \le \, \Pr\big( F^c \,|\, E_j \big)\,,\]
which contradicts our assumption $\Pr\big(F\,|\,E_j\big)\ge 1-\delta$.
\end{proof}

In order to deduce Proposition~\ref{prop:rob} from Lemma~\ref{lem:EF}, we need to bound the variance of the number of $m$-sets $A$ with exactly $k^2 m$ edges of $G(n,p)$, and show that such a set is likely to contain a graph $G$ as in the statement of the proposition. We begin with the latter claim, which follows from a standard argument, originally due to Erd\H{o}s~\cite{Erd59}.

\begin{lemma}\label{lemma:Erdos} For all sufficiently large $k$ the following holds.
Let $H$ be chosen uniformly from the set of graphs with $m$ vertices and $k^2 m$ edges. Then, with high probability as $m\to\infty$, there exists a spanning subgraph $G \subset H$ with 
$$\chi(G) \ge k \qquad \text{and} \qquad \textup{girth}(G) \ge k,$$
and moreover, for every pair $A,B \subset V(H)$ of disjoint vertex sets with $|A|,|B| \ge m / 4k$, the graph $G[A,B]$ contains an edge.
\end{lemma}

\begin{proof}
Given $A$ and $B$ vertex disjoint sets of size exactly $m/4k$, since $e\big( H[A,B] \big)$ is a hypergeometrically distributed random variable with expected value 
$$|A| \cdot |B| \cdot k^2 m \cdot \binom{m}{2}^{-1} \ge \, \frac{m}{16}\,,$$
we have (see for example~\cite[Theorem~2.10]{JLRbook})
\[\Pr\big(e(H[A,B])\le\tfrac{m}{32}\big)\le e^{-m/200}\,.\]
By the union bound, the probability that there exist $A$ and $B$ in $H$ such that $e\big(H[A,B]\big)\le m/32$ is at most
\[\binom{m}{m/4k}^2e^{-m/200}\le e^{-m/300}\]
where the inequality holds since $k$ is sufficiently large.

Now let $Z$ denote the number of cycles in $H$ of length at most $k$. Then 
\[\Ex[Z] \, \le \, \sum_{\ell = 3}^k m^\ell \binom{\binom{m}{2} - \ell}{k^2 m - \ell} \binom{\binom{m}{2}}{k^2 m}^{-1} \le \, \sum_{\ell = 3}^k m^\ell \bigg( \frac{2k^2}{m-1} \bigg)^\ell \le \, (2k^2)^{k+1}\,.\]
By Markov's inequality, it follows that with high probability we have $Z\le\tfrac{m}{64}$, and in particular with high probability $H$ is such that the following holds. By removing one edge from each cycle of length at most $k$, we obtain a subgraph $G$ with $\textup{girth}(G) \ge k$ and such that for every pair $A,B \subset V(H)$ of disjoint vertex sets with $|A|,|B| \ge m / 4k$, the graph $G[A,B]$ contains an edge. Then $G$ has no independent set of size $m/k$, so $\chi(G)\ge k$.
\end{proof}

\begin{lemma}\label{lem:var} 
Let $0 < \eps < 1$ and $k \in \N$, and suppose that $p = p(n)$ satisfies $n^{-1/2}\ll p \ll 1$. Let $X$ be the random variable that counts the number of sets $A\subset [n]$ with $|A| = m = \eps / p$ such that $e\big(G(n,p)[A]\big)=k^2 m$. Then
$$\Var( X ) \, \le \, \frac{Cm^2}{n} \cdot \Ex[X]^2,$$
where $C = e^{O( k^4 / \eps)}$. 
\end{lemma}

\begin{proof}
Let $s=k^2 m$. Observe that the expectation of $X$ is
\begin{equation}\label{eq:varX}
\Ex[X]=\binom{n}{m} \binom{\binom{m}{2}}{s} p^s (1 - p)^{\binom{m}{2} - s}\,.
\end{equation}

We now calculate $\Ex[X^2]$. We need to bound the expected number of pairs $(A,B)$ of sets of size $m$, each with exactly $s = k^2 m$ edges. Let $\ell = |A \cap B|$ denote the size of the intersection of these sets, and let $t = e(A \cap B)$ denote the number of edges contained in both sets. We have
\[\Ex\big[ X^2 \big] \, = \, \sum_{\ell = 0}^m \binom{n}{\ell}\binom{n-\ell}{m-\ell}\binom{n-m}{m-\ell} \sum_{t = 0}^{\binom{\ell}{2}} \binom{\binom{\ell}{2} }{t} \binom{\binom{m}{2} - \binom{\ell}{2}}{s - t}^2 p^{2s - t} (1 - p)^{2\binom{m}{2} - \binom{\ell}{ 2} - 2s + t}\,.\]
Note that $\binom{n}{m-\ell}\le\big(\frac{m}{n-m}\big)^\ell\binom{n}{m}$, since $2m \le n$, and that $\binom{\binom{m}{2}-\binom{\ell}{2}}{s-t} \le
\binom{\binom{m}{2}}{s-t}\le\big(\frac{s}{\binom{m}{2}-s}\big)^t\binom{\binom{m}{2}}{s}$.
Thus
$$\binom{n-\ell}{m-\ell}\binom{n-m}{m-\ell}\le \binom{n }{ m - \ell}^2 \le \left( \frac{2m}{n} \right)^{2\ell} \binom{n}{ m}^2 \quad \text{and} \quad \binom{\binom{m}{ 2} - \binom{\ell }{ 2} }{ s - t} \le \left( \frac{4s}{m^2} \right)^t \binom{\binom{m }{ 2} }{ s}\,,$$
and hence
\begin{align*} 
\Ex\big[ X^2 \big] & \, \le \, \sum_{\ell = 0}^m \binom{n }{ \ell} \left( \frac{2m}{n} \right)^{2\ell}\binom{n }{ m}^2 \sum_{t = 0}^{\binom{\ell }{ 2}} \binom{\binom{\ell }{2}}{ t} \left( \frac{4s}{m^2} \right)^{2t}\binom{\binom{m }{ 2} }{ s}^2 p^{2s - t} (1 - p)^{2\binom{m}{ 2} - \binom{\ell}{ 2} - 2s + t}\\
& \, \eqByRef{eq:varX} \, \Ex[X]^2\sum_{\ell = 0}^m \binom{n}{ \ell} \left( \frac{2m}{n} \right)^{2\ell} \sum_{t = 0}^{\binom{\ell }{ 2}} \binom{\binom{\ell }{ 2}}{ t} \left( \frac{4s}{m^2} \right)^{2t} p^{ - t} (1 - p)^{ - \binom{\ell }{ 2} + t}\,.
\end{align*}
Substituting $\binom{n}{\ell} \le \big( \frac{en}{\ell} \big)^\ell$ and $\binom{\binom{\ell}{2}}{t} \le\big( \frac{e\ell^2}{2t} \big)^t$, and recalling that $p\ell \le pm = \eps$, and thus $(1-p)^{-\ell}\le e^{2p\ell}=O(1)$, we have
\begin{align*}
\Ex[X^2] & \, \le \, \Ex[X]^2\sum_{\ell = 0}^m \bigg(\frac{en}{\ell} \cdot \frac{4m^2}{n^2} \bigg)^{\ell} \sum_{t = 0}^{\binom{\ell }{ 2}} \bigg( \frac{e\ell^2}{2t} \cdot \frac{16s^2}{m^4} \bigg)^{t} p^{ - t} (1 - p)^{ - \binom{\ell }{ 2} + t}\\
& \, \le \, \Ex[X]^2 \sum_{\ell = 0}^m e^{O(\ell)} \bigg(\frac{m^2}{n\ell} \bigg)^{\ell} \sum_{t = 0}^{\binom{\ell}{ 2}} \bigg( \frac{8es^2\ell^2}{m^4tp} \bigg)^{t}\,.
\end{align*}
Since $(x / t)^t$ is maximised when $t = x / e$, and therefore has maximum $e^{x/e}$, it follows that
\[
\Ex[X^2] \, \le \, \Ex[X]^2\sum_{\ell = 0}^m e^{O(\ell)} \bigg(\frac{m^2}{n\ell} \bigg)^{\ell} \exp\bigg(\frac{8s^2\ell^2}{m^4 p} \bigg). % I replaced the O(1) with 8 - PA
\]
Now, since $\ell \le m$, $pm = \eps$ and $s = k^2 m$, we obtain
$$\Ex[X^2] \, \le \, \Ex[X]^2\sum_{\ell = 0}^m e^{O(\ell)} \bigg(\frac{m^2}{n\ell} \bigg)^{\ell} \exp\bigg(\frac{8k^4 \ell}{\eps} \bigg) \le \, \Ex[X]^2\sum_{\ell = 0}^m \bigg( \frac{C' m^2}{n\ell} \bigg)^{\ell}$$
for some $C' = e^{O( k^4 / \eps)}$. This last sum is bounded above by a geometric series with ratio $C' m^2/n$, which, since $m = \eps/p$ and by our choice of $p$, is smaller than $\tfrac12$. The sum is therefore bounded above by $1 + 2C'm^2 / n$, as desired.
\end{proof}

We can now easily deduce Proposition~\ref{prop:rob}. 

\begin{proof}[Proof of Proposition~\ref{prop:rob}]
Given $\eps>0$, we can assume $k\ge4/\eps$ is sufficiently large. Set $m =
\eps / p$,
define $E_A$ to be the event that $G(n,p)[A]$ contains exactly $k^2|A|$ edges, and let $F$ denote the event that there exists a subgraph $G \subset G(n,p)$ such that $v(G) \le \eps / p$, $\chi(G) \ge k$ and $\textup{girth}(G) \ge k$, and moreover, for every pair $A,B \subset V(G)$ of disjoint vertex sets of size at least $\eps |G|$, the graph $G[A,B]$ contains an edge. We are required to prove that $\Pr(F) = 1 - o(1)$. 
But this follows from Lemma~\ref{lem:EF} since the conditions of this lemma
are satisfied for some $\delta=o(1)$. Indeed, since $n^{-1/2} \ll p = o(1)$, we have 
$\Var(X)=o\big( \Ex[X]^2\big)$ by Lemma~\ref{lem:var}, where $X = \sum_{|A| = m} E_A$, and by Lemma~\ref{lemma:Erdos} we have $\Pr\big( F \,|\, E_A \big) = 1 - o(1)$.
\end{proof}

\section{Proof of Theorem~\ref{thm:classhigh}}\label{sec:classhigh}

To prove Theorem~\ref{thm:classhigh}, we will use some known results. The first is the so-called sparse random Erd\H{o}s--Stone theorem, which was originally conjectured by Kohayakawa, \L uczak and R\"odl~\cite{KLR}, and proved by Conlon and Gowers~\cite{ConGow} (for strictly balanced graphs $H$) and Schacht~\cite{SchTuran} (in general). 

\begin{theorem}[Conlon and Gowers, Schacht]\label{thm:CGS}
For every graph $H$, every $\gamma > 0$, and every $p \gg n^{-1/m_2(H)}$, the following holds with high probability. For every $H$-free subgraph $G \subset G(n,p)$, we have
$$e(G) \, \le \, \bigg( 1 - \frac{1}{\chi(H)-1} + \gamma \bigg) p\binom{n }{ 2}.$$
In particular, $\delta_\chi(H,p) \le \pi(H)$. 
\end{theorem}

The bound on~$p$ in Theorem~\ref{thm:CGS} is sharp, as is shown by the following proposition.

\begin{prop}\label{prop:verysmallp}
For every graph $H$, every $\gamma > 0$, and every $\tfrac{\log n}{n} \ll p \ll n^{-1/m_2(H)}$, the following holds with high probability. There exists an $H$-free spanning subgraph $G \subset G(n,p)$ with
\[\delta(G) \ge \big( 1 - 2\gamma \big) pn \qquad \text{and} \qquad \chi(G)\ge\tfrac12\gamma^{-1/2}\,.\]
In particular, $\delta_\chi(H,p) = 1$.
\end{prop}

For the proof we need the following lemma, which is a straightforward consequence of the celebrated `polynomial concentration inequality' of Kim and Vu~\cite{KV}\footnote{An anonymous referee pointed out to us that one can avoid the polynomial concentration machinery by using a very nice trick of Krivelevich~\cite{KriLarge}.}. Recall that a graph $F$ is said to be \emph{$2$-balanced} if $\frac{e(F')-1}{v(F')-2} \le \frac{e(F)-1}{v(F)-2}$ for each subgraph $F' \subsetneq F$ with $v(F') \ge 3$. Given a graph $F$ and a vertex $v$, let $X_F(v)$ denote the random variable that counts the number of copies of $F$ in $G(n,p)$ that contain $v$.

\begin{lemma}\label{lem:Fconc}
Let $F$ be a $2$-balanced graph with $m_2(F) > 1$, and suppose that $p = p(n)$ satisfies $\tfrac{(\log n)^{2e(F)+5}}{n} \le p\le n^{-1/m_2(F)}$. Then
\begin{equation}\label{eq:Fconc}
\Pr\bigg( X_F(v) \ge \Ex\big[ X_F(v) \big] + \frac{pn}{\log n} \bigg) \, \le \, \frac{1}{n^2}
\end{equation}
for all sufficiently large $n$.
\end{lemma}

In order to prove Lemma~\ref{lem:Fconc}, we will need to recall the main theorem of~\cite{KV}. For each set $S \subset E(K_n)$, let $E_F(S,v)$ denote the conditional expectation, given that $S \subset E\big( G(n,p) \big)$, of the number of copies of $F$ in $G(n,p)$ which contain $v$ and use all the edges in $S$. Let $E'_F(v) = \max_{S\neq\emptyset}E_F(S,v)$, and $E_F(v) = \max\big\{ E'_F(v), \, \Ex[X_F(v)] \big\}$. The main theorem of~\cite{KV} implies\footnote{To obtain~\eqref{eq:KV}, we apply the general theorem stated in~\cite{KV} to the polynomial corresponding to the random variable $X_F(v)$, with $\lambda = 2 e(F) \log \binom n 2$.} that
\begin{equation}\label{eq:KV}
 \Pr\Big( X_F(v) \ge \Ex\big[ X_F(v) \big] + (\log n)^{e(F)+1} \big( E_F(v)E'_F(v) \big)^{1/2}  \Big) \, \le \, n^{-2e(F)}
\end{equation}
if $n$ is sufficiently large.

\begin{proof}[Proof of Lemma~\ref{lem:Fconc}]
In order to apply~\eqref{eq:KV} we need to bound $E_F(S,v)$ from above for each non-empty set of edges $S$. We claim that, for each such $S$, we have 
\begin{equation}\label{eq:EFSv:bound}
E_F(S,v) \, \le \, pn(\log n)^{-2e(F)-4}.
\end{equation}
We now justify this claim. Given a non-empty set $S$ of edges of $K_n$, let $T$ be the set of vertices incident to the edges $S$. 
Observe that
 \begin{equation}\label{eq:Fcheck}
  E_F(S,v) \le v(F)^{|T|}n^{v(F)-|T|}p^{e(F)-|S|}\,.
 \end{equation}
Suppose first that $|T|=v(F)$. Then $E_F(S,v) \le v(F)^{|T|} = O(1)$, from which~\eqref{eq:EFSv:bound} follows since $p \ge \tfrac{(\log n)^{2e(F)+5}}{n}$. On the other hand, since $F$ is $2$-balanced, if $2 \le |T| < v(F)$ then we have $|S| \le m_2(F) \big( |T| - 2 \big) + 1$. Since $e(F) = m_2(F)\big( v(F) - 2 \big) + 1$ and $m_2(F) > 1$, it follows that
$$e(F) - |S| \, = \, m_2(F)\big( v(F) - |T| - 1 \big) + 1 + c$$
for some $c > 0$, and hence, from~\eqref{eq:Fcheck}, that
$$E_F(S,v) \le v(F)^{|T|} p^{1+ c} n \cdot \big( p^{m_2(F)} n \big)^{v(F) - |T| - 1} \, \le \, pn (\log n)^{-2e(F)-4}$$
if $n$ is sufficiently large, by our choice of $p$. This proves~\eqref{eq:EFSv:bound}.

Now, to deduce~\eqref{eq:Fconc}, simply note that $\Ex[X_F(v)] \le p^{e(F)} n^{v(F)-1} \le pn$, since $p \le n^{-1/m_2(F)}$ and $F$ is $2$-balanced. Thus, by~\eqref{eq:EFSv:bound}, we have
$$E_F(v)E'_F(v) \, \le \, (pn)^2 (\log n)^{-2e(F)-4},$$
and hence, by~\eqref{eq:KV}, we have  
$$\Pr\bigg( X_F(v) \ge \Ex\big[ X_F(v) \big] + \frac{pn}{\log n} \bigg) \, \le \, n^{-2e(F)},$$
as required.
\end{proof}

We can now easily deduce Proposition~\ref{prop:verysmallp}.

\begin{proof}[Proof of Proposition~\ref{prop:verysmallp}]
Suppose first that $p \ge \tfrac{(\log n)^{2e(H) + 5}}{n}$. In this case let $F \subset H$ be a subgraph of $H$ with $\frac{e(F) - 1}{v(F) - 2} = m_2(H)$, and note that $F$ is 2-balanced, and that $m_2(F) = m_2(H) > 1$, since otherwise the statement is vacuous. Since $p \ll n^{-1/m_2(H)}$ implies that $\Ex\big[ X_F(v) \big] = o(pn)$, it follows by Lemma~\ref{lem:Fconc} that, with high probability, every vertex of $G(n,p)$ lies in at most $\gamma p n$ copies of $F$. 

On the other hand, if $p \le \tfrac{(\log n)^{O(1)}}{n}$, then we may take
$F$ to be an arbitrary cycle in $H$. (If $H$ is a forest then $m_2(H) \le
1$ and again the proposition holds vacuously.) To estimate the number of
pairs of copies of $F$ in $G(n,p)$ both containing $v$, observe that if two
such copies of $F$ intersect in a set of edges $S$ and vertices $T$, then $T$
contains $v$ and thus we have $|S|\le|T|-1$. Furthermore, $e(F)=v(F)$, and so the expected number of pairs of copies of $F$ in $G(n,p)$, both containing $v$, is at most
$$\sum_{S \subsetneq E(F)} O\Big( p^{2e(F) - |S|} n^{2v(F) - |T| - 1} \Big)
\, \le \, \frac{(\log n)^{O(1)}}{n^2}\,.$$
Hence, by Markov's inequality, with high probability every vertex of $G(n,p)$ lies in at most one copy of $F$.

Now, by Chernoff's inequality, $G(n,p)$ has minimum degree at least $(1-\gamma)pn$, and each $X \subset V\big(G(n,p)\big)$ with $|X| \ge\gamma n$ contains more than $p|X|^2 / 4$ edges. If all three of these likely events occur, then we can construct the desired graph $G \subset G(n,p)$ simply by deleting one edge from each copy of $F$ in $G(n,p)$. 

To see that $G$ has the required properties, observe first that $G$ is $F$-free, and therefore $H$-free. Next, note that we have deleted at most $\gamma p n$ edges at each vertex of $G(n,p)$, so we have $\delta(G)\ge(1-2\gamma)pn$, as claimed. Finally, if $|X|\ge 2\sqrt{\gamma}n$ then, since at most $\gamma p n^2$ edges of $G(n,p)$ are deleted to obtain $G$, and $X$ contains more than $\gamma p n^2$ edges of $G(n,p)$, the set $X$ is not independent in $G$. It follows that $\chi(G)\ge \tfrac12\gamma^{-1/2}$ as desired.
\end{proof}

For constant~$p$, on the other hand, the chromatic threshold was determined in~\cite{dense}.

\begin{theorem}\label{thm:pconst}
  For each constant $p > 0$ and graph $H$, we have
  $\delta_\chi(H,p) = \delta_\chi(H)$.
\end{theorem}

This together with the following lower bound on $\delta_\chi(H,p)$
establishes Theorem~\ref{thm:classhigh}.

\begin{theorem}\label{thm:sparsetop} 
Let  $r \ge 4$, $s \in \mathbb{N}$ and $\gamma > 0$. If $p = p(n)$ satisfies $n^{-1/2} \, \ll \, p \, \ll \, 1$ 
as $n \to \infty$, then with high probability $G(n,p)$ contains a spanning subgraph $G$ with
$$\chi(G) \ge s \qquad \text{and} \qquad \delta(G) \ge \bigg( \frac{r-2}{r-1} - \gamma \bigg) pn,$$ 
such that every $s$-vertex subgraph of $G$ is $(r-1)$-colourable.

In particular, $\delta_\chi(H,p) \ge \pi(H)$ for every graph $H$ with $\chi(H) \ge 4$.
\end{theorem}

Before proving this theorem, we first spell out the deduction of Theorem~\ref{thm:classhigh}.

\begin{proof}[Proof of Theorem~\ref{thm:classhigh}]
If $p > 0$ is constant then the result follows from Theorem~\ref{thm:pconst}, and if $\frac{\log n}{n} \ll p \ll n^{-1/m_2(H)}$ then it follows from Proposition~\ref{prop:verysmallp}. Moreover, the upper bound in the range $p \gg n^{-1/m_2(H)}$ follows by Theorem~\ref{thm:CGS}.

It remains to show that $\delta_\chi(H,p) \ge \pi(H)$ for every $n^{-1/m_2(H)} \ll p \ll 1$. If $H$ is bipartite then this is immediate, since $\pi(H) = 0$, so let us assume that $\chi(H) \ge 4$  and $m_2(H) \ge 2$. (Note that $m_2(H) > 2$ for every graph $H$ with $\chi(H) \ge 5$. Indeed, if $m_2(H) \le 2$ then $e(F) \le 2v(F) - 3$ for every subgraph $F \subset H$ with more than one vertex. In particular, this implies that every subgraph of $H$ has a vertex of degree at most 3, and hence $\chi(H) \le 4$.) But now $n^{-1/m_2(H)} \ge n^{-1/2}$, and so the claimed bound follows from Theorem~\ref{thm:sparsetop}. 
\end{proof}

Now we will use Proposition~\ref{prop:rob} to prove Theorem~\ref{thm:sparsetop}.

\begin{proof}[Proof of Theorem~\ref{thm:sparsetop}]

To construct $G$, we first partition the vertex set into sets $X$ and $Y$, with $|X| = n/(r-1)$, and expose the edges of $G(n,p)$ contained in $X$. By Proposition~\ref{prop:rob}, with high probability, there exists a subgraph $F$ on at most $\eps / p$ vertices with girth and chromatic number both at least $s+1$. We fix one such $F$.

We next expose the edges of $G(n,p)$ between $V(F)$ and $Y$, and let $I_u = N(u)\cap Y$ for each $u \in V(F)$. Set 
$$V_1 = \big( X \setminus V(F) \big) \cup \bigcup_{u \in V(F)} I_u$$
and let $V_2 \cup \dots \cup V_{r-1}$ be an arbitrary equipartition of $Y
\setminus \bigcup_{u \in V(F)} I_u$.  Note that we have not yet revealed
any pair between any $V_i$ and $V_j$ with $i \ne j$.
We reveal them now.

We are now ready to define $G$ to be the spanning subgraph of $G(n,p)$ with edge set
$$E(F) \cup \big\{ uv : u \in V(F), \, v \in I_u \big\} \cup \big\{ uv \in E(G(n,p)) : u \in V_i, \, v \in V_j, \, i \ne j \big\}.$$
We claim that, with high probability, $G$ has the desired
properties. Indeed, note first that $\chi(G) \ge \chi(F) > s$.
Next we (implicitly) use several Chernoff bounds to show that~$G$ has
sufficiently high minimum degree.
Indeed, since $p \gg n^{-1/2}$ and $v(F) \le \eps / p$, with high probability we have   
$| \bigcup_{u \in V(F)} I_u | \le pn \cdot v(F) \le \eps n$, which implies
$|V_i|\ge \frac{n}{r-1}-\eps n$.
So with high probability 
every vertex $u\in V_i$ has at least $\big( \frac{r-2}{r-1} - \gamma \big) p n$ neighbours in $\bigcup_{j\neq i} V_j$ and
every vertex $u\in V(F)$ has at least $\big( \frac{r-2}{r-1} - \gamma \big) p n$ neighbours in $Y$ (because the edges
between $V(F)$ and~$Y$ were only revealed after fixing~$F$), as required.

Finally, we claim that $\chi\big( G[W] \big) \le r-1$ for every $s$-vertex subset $W \subset V(G)$. To show this, first colour $W \cap V_i$ with colour $i$ for each $1 \le i \le r-1$. Now, the remaining vertices $W \cap V(F)$ induce a forest in $G$, because $F$ has girth greater than $s$, and their neighbours in $G$ are all in $V_1$, by construction. We can therefore complete a proper colouring of $G[W]$ using only colours $2$ and $3$ inside $F$. Since $r \ge 4$, this completes the proof. 
\end{proof}

We finish this section by noting that Theorem~\ref{thm:classhigh} does not cover all graphs of chromatic number four. 

\begin{prop}\label{prop:inf42}
There exist infinitely many graphs $H$ with $\chi(H) = 4$ and $m_2(H) < 2$. 
\end{prop}

\begin{proof}
We will construct a graph $G_0$, which we call a \emph{gadget} (see Figure~\ref{fig:Simon}), with 12 vertices and 19 edges, and which satisfies the following key property: there are (non-adjacent) vertices $x,y \in V(G_0)$ such that in any proper 3-colouring of $G_0$ the vertices $x$ and $y$ receive different colours. In other words, if we replace an edge $xy$ of a larger graph $G$ by a gadget on $xy$, then (when 3-colouring $G$) the gadget plays the same role as the edge.

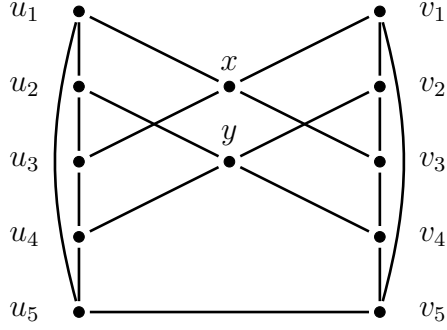
\begin{figure}
\begin{tikzpicture}
  \tikzset{label distance=3mm}
  \foreach \a in {1,...,5} { 
    \wvertex[$u_\a$]{u\a}{0,-\a} 
    \evertex[$v_\a$]{v\a}{4,-\a} 
  }
  \tikzset{label distance=0mm}
  \nvertex[$x$]{x}{2,-2}
  \nvertex[$y$]{y}{2,-3}
  \foreach \a/\b in {1/2,2/3,3/4,4/5} {
    \edge{u\a}{u\b}
    \edge{v\a}{v\b}
  }
  \foreach \a in {1,3} {
    \edge{x}{u\a}
    \edge{x}{v\a}
  }
  \foreach \a in {2,4} {
    \edge{y}{u\a}
    \edge{y}{v\a}
  }
  \edge{u5}{v5}
  \draw[edge] (u1) to [bend right=15] (u5);
  \draw[edge] (v1) to [bend left=15] (v5);
\end{tikzpicture}
\caption{the gadget~$G_0$}
\label{fig:Simon}
\end{figure}

To define the gadget, set $V(G_0) = \{x,y\} \cup \{u_i, v_i : 1 \le i \le 5\}$ and
\begin{align*}
& E(G_0) \, = \, \big\{ u_1u_2, u_2u_3, u_3u_4, u_4u_5, u_5u_1 \big\} \cup \big\{ v_1v_2, v_2v_3, v_3v_4, v_4v_5, v_5v_1 \big\} \\
& \hspace{4cm} \cup \big\{ xu_1, xu_3, xv_1, xv_3 \big\}  \cup \big\{ yu_2, yu_4, yv_2, yv_4 \big\} \cup \big\{ u_5v_5 \big\}.
\end{align*}
Now, let $c \colon V(G_0) \to \{1,2,3\}$ be a proper colouring, and suppose that $c(x) = c(y) = 1$. Then none of the vertices of $\{ u_i, v_i : 1 \le i \le 4 \}$ receives colour 1, and therefore it follows that $c(u_5) = c(v_5) = 1$, a contradiction. 
 
Now, given a graph $H$ with $\chi(H) = 4$, we can construct a graph $H'$ with 
$$\chi(H') = 4, \qquad v(H') = v(H) + 10 \qquad \text{and} \qquad e(H') = e(H) + 18,$$    
simply by replacing any edge of $H$ by a copy of the gadget. Moreover, if $m_2(H) < 2$ then one easily checks 
that $m_2(H') < 2$. Thus, in order to construct the claimed infinite family of graphs, it suffices to construct a single example. 

In order to do so, consider the graph $H$ obtained by replacing every edge
$xy$ of $K_4$ by a copy of the gadget. The resulting graph has $6\cdot
12-4\cdot 2=64$ vertices and $6\cdot 19=114$ edges. Moreover, it may be easily checked that $m_2(H) < 2$, as required. 
\end{proof}

\section{Lower bounds for Theorems~\ref{thm:K3} and~\ref{thm:C5}}
\label{sec:lower}

By Theorem~\ref{thm:pconst} we have that $\delta_\chi(K_3,p)=\frac13$ and
$\delta_\chi(C_5,p)=0$ for any constant $p>0$, and by Proposition~\ref{prop:verysmallp} we have $\delta_\chi(H,p) = 1$ whenever $\tfrac{\log n}{n} \ll p \ll n^{-1/m_2(H)}$. In this section we will
provide three constructions (see Propositions~\ref{prop:cloud:lower},
\ref{prop:thundercloud:lower} and~\ref{prop:C5:simon}) which imply the
remaining lower bounds for $\delta_\chi(K_3,p)$ and $\delta_\chi(C_5,p)$ in
Theorems~\ref{thm:K3} and~\ref{thm:C5}.
However, since these constructions are also needed in~\cite{dense}, we shall, instead of giving them for $H=K_3$ and $H=C_5$, provide them in the following more general setting.

\begin{defn}\label{def:cloudforest}
A graph $H$ is a \emph{cloud-forest graph} if there is an independent set $I \subset V(H)$ (the \emph{cloud}) such that $V(H) \setminus I$ induces a forest $F$, the only edges from $I$ to $F$ go to leaves or isolated vertices of $F$, and no two adjacent leaves in $F$ send edges to $I$. 

Moreover, $H$ is a \emph{thundercloud-forest graph} if there is a cloud $I \subset V(H)$, which witnesses that $H$ is a cloud-forest graph, such that every odd cycle in $H$ uses at least two vertices of~$I$.
\end{defn}

Note that $K_3$ is not a cloud-forest graph, $C_5$ is a cloud-forest but not a thundercloud-forest graph, and $C_{2k+1}$ is a thundercloud-forest graph for every $k \ge 3$. 

The following proposition implies $\delta_\chi(K_3,p)\ge\frac12$ for $n^{-1/2} \ll p(n) = o(1)$.

\begin{prop}\label{prop:cloud:lower}
Let $H$ be a graph with $\chi(H) = 3$, and suppose that $n^{-1/2} \ll p(n) = o(1)$. If $H$ is not a cloud-forest graph, then
$\delta_{\chi}(H,p) \ge \frac{1}{2}$.
\end{prop}

Similarly, the next proposition implies $\delta_\chi(C_5,p)\ge\frac13$ for $n^{-1/2} \ll p(n) = o(1)$.

\begin{prop}\label{prop:thundercloud:lower}
Let $H$ be a graph with $\chi(H) = 3$, and suppose that $n^{-1/2} \ll p(n) = o(1)$. If $H$ is not a thundercloud-forest graph, then
$\delta_{\chi}(H,p) \ge \frac{1}{3}$.
\end{prop}

The constructions that prove these propositions are similar to (though
somewhat more complicated than) that in the proof of
Theorem~\ref{thm:sparsetop} and rely, again, on Proposition~\ref{prop:rob}. 

\begin{proof}[Proof of Proposition~\ref{prop:cloud:lower}]
We will show that, for every $s \in \N$ and $\gamma > 0$, with high probability $G(n,p)$ contains a spanning subgraph $G$ with 
$$\chi(G) \ge s \qquad \text{and} \qquad \delta(G) \ge\bigg( \frac{1}{2} - \gamma \bigg)pn$$ 
such that every $s$-vertex subgraph of $G$ is a cloud-forest graph. This fact implies that $\delta_{\chi}(H,p) \ge 1/2$ for every graph $H$ that is not a cloud-forest graph.

The construction of $G$ is similar to that in the proof of Theorem~\ref{thm:sparsetop}, except that we will need to ensure that the neighbourhoods (in $G$) of the vertices of $F$ are disjoint. To be precise, let us partition the vertex set into sets $X$ and $Y$, with $|X| = |Y| = n/2$, and expose the edges of $G(n,p)$ contained in $X$. With high probability, we obtain (by Proposition~\ref{prop:rob}) a subgraph $F$ on at most $\eps / p$ vertices with girth and chromatic number both at least $s+1$. Next, expose the edges of $G(n,p)$ between the vertices of $F$ and $Y$, and for each $u \in V(F)$, define 
$$I_u \, = \, \Big\{ v \in Y : N(v) \cap V(F) = \{ u \} \Big\},$$
and note that the sets $I_u$ are pairwise disjoint. Observe also that $|I_u|$ is a binomial random variable with expected size $|Y| \cdot p(1-p)^{v(F)-1} \ge \big( 1/2 - \eps \big) pn$, and therefore with high probability $|I_u| \ge (1/2 - \gamma) pn$ for every $u \in V(F)$. Now let~$G$ be the graph with edge set
$$E(F) \cup \big\{ uv : u \in V(F), \, v \in I_u \big\} \cup \big\{ uv \in E(G(n,p)) : u \in V_1, \, v \in V_2 \big\},$$
where 
$$V_1 = \big( X \setminus V(F) \big) \cup \bigcup_{u \in V(F)} I_u \qquad \textup{and} \qquad V_2 := Y \setminus \bigcup_{u \in V(F)} I_u.$$
Observe that $\chi(G) \ge \chi(F) > s$, and that $\delta(G) \ge \big( \frac{1}{2} - \gamma \big) p n$ with high probability.

It remains to show that for any set $W \subset V(G)$ of at most $s$
vertices, $G[W]$ is a cloud-forest graph. To do so, we must find an
independent set $I \subset W$ such that $G[W \setminus I]$ is a forest, and
the only edges from $I$ to $W \setminus I$ go to non-adjacent leaves or
isolated vertices of this forest. We claim that this holds for $I = W \cap
V_2$, which is an independent set by the definition of~$G$. Indeed, observe
first that $G[W\setminus I]$ is a forest, since $W\setminus I\subset
V(F)\cup V_1$, the girth of $F$ is greater than $s$, and since each vertex
of $V_1$ has at most one neighbour in $V(F)\cup V_1$. Now, every edge from $I$ meets $W \setminus I$ in $V_1$, and (as just noted) every vertex of $W \cap V_1$ is either an isolated vertex of $G[W\setminus I]$, or a leaf. Since $V_1$ is an independent set in $G$, all of these leaves are non-adjacent, and hence $G[W]$ is a cloud-forest graph, as required.
\end{proof}

In order to prove Proposition~\ref{prop:thundercloud:lower}, we will make use of the following construction of {\L}uczak and Thomass\'e~\cite{LT}, see also~\cite[Theorem~14]{ChromThresh}.

\begin{lemma}[{\L}uczak and Thomass\'e]\label{lem:acycstruct} 
For each $s \in \N$ and $\gamma > 0$, there exists a graph $H$ with 
$$\delta(H) \ge \bigg( \frac{1}{3} - \gamma \bigg) \cdot |H|,$$ 
and a partition $V(H)=A \dcup B \dcup C$ with the following properties:
\begin{itemize}
\item[$(a)$] $|A| \le \gamma |H|$ and $\chi\big( H[A] \big) \ge s$. 
\item[$(b)$]  $B$ and $C$ are independent sets in $H$. 
\item[$(c)$] $H[A,C]$ is empty and $H[B,C]$ is complete. 
\item[$(d)$] Every odd cycle in $H$ on $s$ or fewer vertices uses at least two vertices of $B$.
\end{itemize}
\end{lemma}

\begin{proof}[Proof of Proposition~\ref{prop:thundercloud:lower}]
We will show that, for every $s \in \N$ and $\gamma > 0$, with high probability $G(n,p)$ contains a spanning subgraph $G$ with 
$$\chi(G) \ge s \qquad \text{and} \qquad \delta(G) \ge\bigg( \frac{1}{3} - 3\gamma \bigg)pn$$ 
such that every $s$-vertex subgraph of $G$ is a thundercloud-forest graph. The existence of such a graph $G$ implies that the chromatic threshold with respect to $p$ of every graph that is not a thundercloud-forest graph is at least $1/3$, as required.

We choose $\eps>0$ sufficiently small, and partition the vertex set into
sets $X$ and $Y$, with $|X|=2n/3$ and $|Y|=n/3$. As before, we reveal the
edges within $X$ and use Proposition~\ref{prop:rob} to find a subgraph $F$
with chromatic number and girth greater than $s$, on $\eps/p$ vertices,
such that any two disjoint subsets of $V(F)$ of size at least $\eps|V(F)|$
have an edge joining them. As before, we reveal the edges from $V(F)$ to
$Y$ and for each $u\in V(F)$ let $I_u\subset Y$ be the vertices of $Y$
whose unique neighbour in $V(F)$ is $u$. Now set, slightly differently than before,
$$V_1 = X \setminus V(F) \qquad \text{and} \qquad V_2 = Y \setminus \bigcup_{u \in V(F)} I_u.$$ 
Let $H$ be the graph whose existence is guaranteed by
Lemma~\ref{lem:acycstruct}, and $V(H) = A \cup B \cup C$ be the partition
for which properties~$(a)$--$(d)$ hold. Let 
$$\phi \colon V(F) \cup V_1 \to A\cup B$$ 
be a function which maps $v(F) / |A|$ vertices of $F$ to each element of $A$ and $|V_1|/|B|$  vertices of $V_1$ to each element of $B$.\footnote{It is easy to see that we can adjust slightly the sizes of the sets so that all of these are integers. Alternatively, we may allow some elements to receive an extra vertex.}

We are now ready to define $G$. It is the spanning subgraph of $G(n,p)$ with edge set
\begin{multline*}
\Big\{ uv \in E(F) : \phi(u)\phi(v) \in E(H) \Big\} \cup \Big\{ uv : u \in V(F), v \in I_u \Big\} \\
\cup \Big\{ vw \in E(G(n,p)) : v \in I_u \text{ for some } u \in V(F), \, w \in V_1, \, \phi(u)\phi(w) \in E(H) \Big\} \\
\cup \Big\{ uv \in E(G(n,p)) : u \in V_1, v \in V_2 \Big\}\,.
\end{multline*}
We claim that, with high probability,
$$\delta(G) \, \ge \, \bigg( \frac{1}{3} - 3\gamma \bigg) p n.$$
Indeed, similarly as before, by several applications of a Chernoff bound,
with high probability, every vertex $u \in V(F)$ has at least $|I_u|\ge\big( \frac{1}{3} - \gamma
\big)pn$ neighbours,
every vertex of $V_1$ has at least $\big(\tfrac13 - \gamma\big)pn$
 neighbours in $V_2$, 
and every vertex of $V_2$ has at least $\big(\tfrac23 - \gamma\big)pn$
neighbours in $V_1$. Finally, if $u \in V(F)$ then, since $\delta(H) \ge
\big( \frac13 - \gamma\big) v(H)$ and $|A|\le\gamma|H|$ (and in~$H$ all neighbours of $\phi(u)$
are in $A\cup B$), we have 
$$\big| \big\{ w \in V_1 : \phi(u)\phi(w) \in E(H) \big\} \big| \, \ge \, \bigg( \frac13 - 2\gamma\bigg) n,$$
and hence every vertex $v \in I_u$ has at least $\big( \frac13 - 3\gamma
\big)pn$ neighbours $w \in V_1$ such that $\phi(u)\phi(w) \in
E(H)$. 

We next claim that $\chi(G) \ge \chi\big( G[V(F)] \big) \ge s$. Indeed,
suppose that we colour $G[V(F)]$ with fewer than $s$ colours, and consider
the colouring of $H[A]$ in which the vertex $a \in A$ receives a most
common colour in $\phi^{-1}(a)$. This is a colouring of $H[A]$ with fewer
than $s$ colours, so by Lemma~\ref{lem:acycstruct} there is a monochromatic
edge $aa'$ of $H[A]$. Let $c$ be the colour of $a$ and $a'$, and observe
that there exist sets $Z \subset \phi^{-1}(a)$ and $Z' \subset
\phi^{-1}(a')$ of colour $c$, each of size at least $v(F)/(s|A|) > \eps
v(F)$. Thus, since~$F$ is the subgraph obtained from Proposition~\ref{prop:rob}, the graph $F[Z,Z']$ contains an edge, and since $aa' \in E(H)$, this edge is also present in $G$. It follows that our colouring of $G$ is not proper, and so $\chi\big( G[V(F)] \big) \ge s$, as claimed.

Finally, we claim that $G[W]$ is a thundercloud-forest graph for every set $W \subset V(G)$ with $|W| = s$. To prove this, we need to show that there exists a cloud $I \subset W$ which witnesses that $H$ is a cloud-forest graph, and is such that every odd cycle in $G[W]$ uses at least two vertices of $I$. We will show that this is the case for the set
$$I = W \cap V_1.$$
Note first that $I$ is an independent set in $G[W]$, since $V_1$ is an
independent set in $G$. Next, observe that $G[W \setminus I]$ is a forest,
since the girth of $F$ is greater than $s$, and since each vertex of
$W\setminus \big(I\cup V(F)\big)$ is either in $V_2$ and hence isolated in
$G[W\setminus I]$, or is in some $I_u$ so has at most~$u$ as a neighbour in $G[W\setminus I]$. Moreover, in the latter case $u$ is not adjacent to any vertex of $I$, so that neighbours of $I$ are non-adjacent leaves of the forest $G[W\setminus I]$, and thus~$I$ is a cloud.

Finally, we need to check that every odd cycle in $G[W]$ uses at least two
vertices of $I$, so let $S$ be an odd cycle in $G[W]$. Note first that $S
\not\subset V(F) \cup \bigcup_u I_u$ since $F$ has girth greater than $s$ and each vertex of $\bigcup_u I_u$ has only one neighbour in $V(F)$, and that every cycle containing a vertex of $V_2$ uses at least two vertices of $V_1$. 

Thus the only remaining case we need to rule out is that $S$ consists of one vertex $x \in V_1$, whose neighbours in $S$ are the vertices $y,z \in \bigcup_u I_u$, and a path $P$ with an even number of vertices in $F$ from $u$ to $v$, where $y \in I_u$ and $z \in I_v$. Since $xy$ and $xz$ and $P$ are all edges of $G$, it follows that $\phi(u)\phi(x)$, $\phi(v)\phi(x)$ and $\phi(P)$ are all edges of $H$, which gives us a circuit of odd length in $H$. This circuit contains an odd cycle, which must (by Lemma~\ref{lem:acycstruct}) use at least two vertices of $B$, a contradiction (since $\phi^{-1}(B) = V_1$). This proves that $G[W]$ is indeed a thundercloud-forest graph, as required.
\end{proof}

We end the section with a slightly different (and easier) construction, which implies the remaining lower bounds in Theorems~\ref{thm:C5} and~\ref{thm:Clong}.

\begin{prop}\label{prop:C5:simon} 
Let $k \ge 2$, and suppose that $\frac{(\log n)^{2}}{n} \le p(n) \ll n^{-(2k-3)/(2k-2)}$. Then  $$\delta_{\chi}(C_{2k+1},p) \ge \frac{1}{2}.$$
\end{prop}
\begin{proof}[Proof of Proposition~\ref{prop:C5:simon}]
Let $\omega=\omega(n)$ be any function tending to infinity sufficiently
slowly as $n \to \infty$ and assume that
\begin{equation}\label{eq:C5p}
\frac{(\log n)^{2}}{n} \, \le \, p = p(n) \, \le \, \frac{1}{\omega^3} \cdot n^{-(2k-3)/(2k-2)}.
\end{equation}
Given $s \in \N$ and $\gamma > 0$, we construct, with high probability, a
$C_{2k+1}$-free spanning subgraph $G \subset G(n,p)$ with $\chi(G) \ge s$ and
$\delta(G) \ge \big( \frac{1}{2} - 2\gamma \big) pn$ as follows.

Partition the vertices into sets $X$ and $Y$ with $|X| = |Y| = n/2$, and expose first the edges of $G(n,p)$ contained in a subset~$X' \subset X$ of size $\omega/p$. We claim that, with high probability, there exists a subgraph $F$ of $G(n,p)$ in $X'$ with maximum degree at most $2\omega$, girth at least $3k$, and chromatic number at least $\log\log\omega$. Indeed, this follows simply by removing vertices of degree greater than $2\omega$, and a vertex from each cycle of length at most~$3k$. To spell out the details, observe that the expected number of independent sets in~$X'$ of size $|X'| / \sqrt{\log\omega}$ tends to zero as $n \to \infty$, the expected number of cycles in~$X'$ of length at most $3k$ is at most $3k\omega^{3k}$, and the expected number of vertices of degree greater than $2\omega$ is (by the Chernoff bound) at most $e^{-\omega/3} |X'|$. Applying Markov's inequality with each of these estimates in turn, we see that with high probability there are no independent sets in $X'$ of size $|X'|/\sqrt{\log\omega}$, the number of cycles of length at most $3k$ is less than $3k\omega^{4k}$, and the number of vertices of degree greater than $2\omega$ is at most $e^{-\omega/6}|X'|$. If each of these events occurs, then, since $\omega$ grows sufficiently slowly, the number of vertices of $X'$ we must remove in order to remove all vertices of degree greater than $2\omega$ and cycles of length at most $3k$ is less than $|X'|/2$ for all large $n$. Now any graph with at least $|X'|/2$ vertices and independence number less than $|X'|/\sqrt{\log\omega}$ has chromatic number at least $\log\log n$, so we obtain the desired $F$.

We now give the pairs of vertices lying between $X$ and $Y$ an order $\tau$, in which the pairs with one end in $V(F)$ come first, but which is otherwise arbitrary. We obtain a spanning subgraph $G$ of $G(n,p)$ by the following process. We start with $E(G_0)=E(F)$, and then for each $1\le i\le |X||Y|$ we define $G_i$ as follows. Let $xy$ be the $i$th edge in $\tau$. If $xy \in E\big( G(n,p) \big)$, and $G_{i-1}\cup \{xy\}$ does not contain a copy of $C_{2k+1}$, and $\deg_{G_{i-1}\cup\{xy\}}\big(y,V(F)\big)\le 2\omega$, then we set $G_i=G_{i-1}\cup\{xy\}$. Otherwise we set $G_i=G_{i-1}$. We let $G=G_{|X||Y|}$.

We claim that $G$ is the desired graph. Indeed, $G_0$ certainly contains no copy of $C_{2k+1}$, and has the desired chromatic number, so by construction the same is true of $G$. It remains to show that with high probability $\delta(G)\ge\big(\tfrac12 - 2\gamma\big)pn$. We will bound the degrees of vertices in $Y$, then in $X \setminus V(F)$, and finally in $V(F)$. Observe first that, by Chernoff's inequality (and since $p\ge\tfrac{(\log n)^2}{n}$), the following events holds with high probability, where $N(x)$ is the neighbourhood in $G(n,p)$ of the vertex $x$: 
\begin{itemize}
\item[$(a)$] $|N(x) \cap Y| \in \big(\tfrac12 \pm \gamma\big)pn$ for each $x \in X$, 
\item[$(b)$] $|N(y) \cap X|  \in \big(\tfrac12 \pm \gamma\big)pn$ for each $y \in Y$, and 
\item[$(c)$] there are at most $e^{-\omega/10}n$ vertices $y\in Y$ such that $|N(y) \cap V(F)| \ge 2\omega$. 
\end{itemize}
Let us assume that each of these likely events holds.

We claim that, for each $y \in Y$, the number of edges $xy$ of $G(n,p)$, with $x\in X$, which are not present in $G$ is stochastically dominated by $\Bin(\omega/p,p) + \Bin\big(8k\omega^2(pn)^{2k-2},p\big)$. Indeed, the number of edges of $G(n,p)$ between $y$ and $V(F)$ is dominated by $\Bin(\omega/p,p)$, and at worst we remove all such edges. The number of paths leaving $y$ of length $2k$ is at most $8k\omega^2(pn)^{2k-2}$, since in constructing such a path we never have more than $pn$ choices for the next vertex, and at some point we must see a sequence of a vertex in $Y$ followed by two vertices in $F$, and we have at most $2\omega$ choices for each vertex of $F$. Not all of these paths need be in any given $G_{i-1}$, but since there are certainly not more, since the $G_j$ form a nested sequence of graphs, and since the event that the $i$th edge of $\tau$ appears in $G(n,p)$ is independent of $G_{i-1}$, the claimed stochastic domination follows. By Chernoff's inequality and the union bound, and by our choice of $p$, it follows that, with high probability, for each $y \in Y$ there are at most 
$$\gamma p n \, \ge \, 2\omega\log n + 16k \omega^2 p(pn)^{2k-2}$$ 
edges of $G(n,p)$ incident to $y$ not present in $G$, as required. 

The proof is almost the same for vertices in $X\setminus V(F)$. Indeed, given $x \in X\setminus V(F)$, it follows exactly as above that the number of edges $xy$ of $G(n,p)$, with $y \in Y$, which are not present in $G$ is stochastically dominated by $\Bin\big(8k\omega^2(pn)^{2k-2},p\big)$. (Note that in this case edges are only removed if they complete a copy of $C_{2k+1}$.) As before, it follows that, with high probability, at most $\gamma p n$ edges of $G(n,p)$ incident to $x$ are not included in $G$ for each $x \in X\setminus V(F)$. 

Finally, for each $x \in V(F)$ we claim that the number of edges $xy$ of $G(n,p)$, with $y \in Y$, which are not present in $G$ is stochastically dominated by $\Bin\big(e^{-\omega/10}n,p\big) + \Bin\big((2\omega pn)^k,p\big)$. The first term here counts edges $xy$ which are in $G(n,p)$ but not added to $G$ because $y$ had too many neighbours in $V(F)$. The set of vertices in $Y$ whose  $G_{i-1}$-neighbourhood in $V(F)$ has size $\lfloor2\omega\rfloor$ is increasing in $i$. Since the probability that a given vertex of $Y$ has at least $\lfloor2\omega\rfloor$ $G(n,p)$-neighbours in $V(F)$ is by Chernoff's inequality at most $e^{-\omega/6}$, by Markov's inequality with high probability the total number of such vertices is at most $e^{-\omega/10}n$. Now since the event that the $i$th edge $xy$ of $\tau$ appears in $G(n,p)$ is independent of $G_{i-1}$, the number of edges $xy$ not included in $G$ because $y$ has too many neighbours in $V(F)$ is stochastically dominated by $\Bin\big(e^{-\omega/10}n,p\big)$. Similarly, assuming that $xy$ is the $i$th edge of $\tau$, the number of vertices $y \in Y$ such that there is a path of length $2k$ from $x$ to $y$ in $G_{i-1}$ is at most $(2\omega)^k(pn)^k$, since (by the definition of $\tau$) every edge of such a path must be incident to $F$. As before, this implies the claimed stochastic domination. By Chernoff's inequality and the union bound, and by our choice of $p$, it follows that, with high probability, for each $x \in V(F)$ there are at most 
$$\gamma p n \, \ge \, e^{-\omega/20} pn + (4\omega)^k p^{k+1} n^k$$ 
edges of $G(n,p)$ incident to $x$ not present in $G$, as required. 
\end{proof}

\section{Upper bounds for Theorems~\ref{thm:K3}, \ref{thm:C5}, and~\ref{thm:Clong}}
\label{sec:upper}

In this section we will determine $\delta_\chi(K_3,p)$ and $\delta_\chi(C_5,p)$ for almost all functions $p = p(n)$, and $\delta_\chi(C_{2k+1},p)$ for a certain range of $p$. Indeed, for $K_3$ the pieces are all already in place. 

\begin{proof}[Proof of Theorem~\ref{thm:K3}]
If $p > 0$ is constant then the result follows from Theorem~\ref{thm:pconst} and the fact that $\delta_\chi(K_3) = 1/3$. Suppose next that $n^{-1/2} \ll p(n) = o(1)$, and recall that, by Theorem~\ref{thm:CGS}, we have $\delta_\chi(H,p) \le 1/2$. On the other hand, recall that~$K_3$ is not a cloud-forest graph, so, by Proposition~\ref{prop:cloud:lower}, we have $\delta_\chi(H,p) = 1/2$ in this range. Finally, if $\tfrac{\log n}{n} \ll p \ll n^{-1/2}$ then it follows from Proposition~\ref{prop:verysmallp} that $\delta_\chi(H,p) = 1$, as required.
\end{proof}

Next, we turn to the case $H = C_5$. Recall that $C_5$ is not a thundercloud-forest graph. We have therefore already proved the following bounds:
\begin{equation*}
\delta_{\chi}(C_5,p) \left\{
\begin{array}{cll}
= 0 & \text{if } & p > 0 \text{ is constant, by Theorem~\ref{thm:pconst}, and since $\delta_\chi(C_5) = 0$,}\smallskip\\ 
\ge \tfrac{1}{3} & \text{if } & n^{-1/2} \ll p \ll 1, \text{ by Proposition~\ref{prop:thundercloud:lower},}\smallskip\\
= \tfrac{1}{2} & \text{if } & n^{-3/4} \ll p \ll n^{-1/2}, \text{ by Theorem~\ref{thm:CGS} and Proposition~\ref{prop:C5:simon},}\smallskip\\
= 1 & \text{if } & \frac{\log n}{n} \ll p \ll n^{-3/4}, \text{ by Proposition~\ref{prop:verysmallp}.}
\end{array} \right.
\end{equation*} 
It therefore only remains to prove that $\delta_{\chi}(C_5,p) \le 1/3$ for all $n^{-1/2} \ll p = o(1)$. We remark that, under the stronger assumption $p = n^{-o(1)}$, this result follows from a general result (for all cloud-forest graphs) proved in~\cite{dense}; the proof of the following result is similar, but significantly simpler.

\begin{prop}\label{prop:C5:thirdupper} 
$\delta_{\chi}(C_5,p) \le 1/3$ for every function $n^{-1/2} \ll p = o(1)$.
\end{prop}

To prove Proposition~\ref{prop:C5:thirdupper} we will use the sparse minimum degree form of Szemer\'edi's Regularity Lemma. First recall the following definitions.

\begin{defn} [$(\eps,p)$-regular pairs and partitions, the reduced graph]
Given a graph $G$ and $\eps, p > 0$, a pair of vertex sets $(A,B)$ is said to be \emph{$(\eps,p)$-regular} if 
$$\bigg| \frac{e\big( G[A,B] \big)}{|A| |B|} - \frac{e\big( G[X,Y] \big)}{|X| |Y|} \bigg| \, < \, \eps p$$ 
for every $X \subset A$ and $Y \subset B$ with $|X| \ge \eps |A|$ and $|Y| \ge \eps |B|$. 

A partition $V(G) = V_0 \cup V_1 \cup \ldots \cup V_k$ is said to be $(\eps,p)$-regular if $|V_0| \le \eps n$, $|V_1| = \ldots = |V_k|$, and at most $\eps k^2$ of the pairs $(V_i,V_j)$ with $1 \le i < j \le k$ are not $(\eps,p)$-regular. 

The \emph{$(\eps,d,p)$-reduced graph} of an $(\eps,p)$-regular partition is the graph $R$ with vertex set  
$V(R) = \{1,\ldots,k\}$ and edge set 
$$E(R) \, = \, \big\{ ij : (V_i,V_j) \text{ is an $(\eps,p)$-regular pair and } e\big( G[V_i,V_j] \big) \ge dp|V_i||V_j|| \big\}.$$ 
\end{defn}

We will use the following form of the Regularity Lemma, see~\cite[Lemma~4.4]{BipartBU} for a proof. Note that although the statement there only guarantees lower-regularity of pairs in the partition, the proof explicitly gives an $(\eps,p)$-regular partition.

\begin{SzRLdeg}
Let $\delta,d,\eps > 0$, $k_0 \in \N$ and $p = p(n) \gg (\log n)^4 / n$. There exists $k_1 = k_1(\delta,\eps,k_0) \in \N$ such that the following holds with high probability. If $G \subset G(n,p)$ has minimum degree $\delta(G) \ge \delta pn$, then there is an $(\eps,d,p)$-regular partition of $G$ into $k$ parts, where $k_0\le k\le k_1$, whose $(\eps,d,p)$-reduced graph $R$ has minimum degree at least $(\delta-d-\eps)k$.
\end{SzRLdeg}

In the proof of Proposition~\ref{prop:C5:thirdupper} we will use the following fact, which is an easy consequence of Chernoff's inequality: if $p \gg n^{-1/2}$, then with high probability there are $\big( 1 + o(1) \big) p|U||V|$ edges between $U$ and $V$ for every pair of sets $(U,V)$ with $|U| = \Omega(pn)$ and $|V| = \Omega(n)$.

\begin{proof}[Proof of Proposition~\ref{prop:C5:thirdupper}]
Let $\gamma > 0$, and let $G$ be a $C_5$-free spanning subgraph of $G(n,p)$ with
$$\delta(G) \ge \bigg( \frac{1}{3} + 2\gamma \bigg) pn.$$
Applying the sparse minimum degree form of Szemer\'edi's regularity lemma to $G$,
with $k_0 = 1 / \gamma$, $d = \gamma /
  10$ and $\eps > 0$ sufficiently small,
 we obtain a partition $V(G) = V_0 \cup V_1 \cup \cdots \cup V_k$, such that the $(\eps,d,p)$-reduced graph $R$ satisfies
\begin{equation}\label{eq:deltaR}
\delta(R) \ge \bigg( \frac{1}{3} + \gamma \bigg) k.
\end{equation}
Given any vertex $v$, we define the \emph{robust second neighbourhood} $N^*_2(v)$ of $v$ in $G$ to be the set of vertices $w$ in $G$ such that $v$ and $w$ have at least $dp^2n$ common neighbours in $G$. We define
\begin{equation}\label{def:Xi:NR2}
X_i := \Big\{ v \in V(G) : \big| N^*_2(v) \cap V_i \big| \ge \big( \tfrac12 + d \big) |V_i| \Big\}
\end{equation}
for each $i \in [k]$, and set $X_0:=V(G)\setminus\big(X_1\cup\dots\cup
X_k\big)$. We will show that $X_i$ is an independent set for each $1 \le i
\le k$, and that $X_0 = \emptyset$. This implies that $\chi(G)$ is bounded
(by a constant depending on~$\gamma$) as required.

Indeed, let us first fix $1 \le i \le k$, and suppose that $uv \in
E(G[X_i])$. Note that, by definition of~$X_i$, the intersection $N^*_2(u)\cap N^*_2(v)$ is non-empty, so let $w \in N^*_2(u) \cap N^*_2(v)$. It follows that 
$$\min\big\{ |N(u)\cap N(w)|, \, |N(v)\cap N(w)| \big\} \, \ge \, dp^2n \, > \, 4,$$
since $p \gg n^{-1/2}$, and therefore there exist distinct vertices $u'$ and $v'$ such that
$$u' \in N(u) \cap N(w) \qquad \text{and} \qquad v' \in N(v) \cap N(w).$$
Since $uv$ is an edge of $G$, it follows that $uu'wv'v$ is a copy of $C_5$ in $G$, which is a contradiction.

To show that $X_0$ is empty, we will use the following claim.

\medskip
\noindent \textbf{Claim:} For each $u \in X_0$, there exists an $(\eps,d,p)$-regular pair $(V_i,V_j)$ such that 
$$|N^*_2(u)\cap V_i| \ge \eps |V_i| \qquad \textup{and} \qquad |N^*_2(u) \cap V_j| \ge \eps |V_j|.$$

\begin{claimproof}[Proof of claim]
The claim follows from some straightforward edge-counting, together with
the minimum degree conditions. Indeed, recalling that $\delta(G) \ge \big(
\frac{1}{3} + 2\gamma \big) pn$, let $U \subset N(u)$ be an arbitrary
subset of size $|U| = pn/3$, and observe that (with high probability) there
are at least $p^2 n^2 / 9$ edges incident to $U$, since $e(G[U]) =
o(p^2n^2)$ (by Chernoff's inequality) and each vertex in $U$ has at least $\big(\frac{1}{3} + 2\gamma \big) pn$ neighbours.

Now, if $|N^*_2(u)\cap V_i| \le \eps |V_i|$, then there are (with high probability) at most 
$$2p  \cdot |U| \cdot \eps|V_i| + dp^2 n \cdot |V_i| \, \le \, 2d p^2 n\cdot |V_i|$$ 
edges between $V_i$ and $U$, so almost all of the edges incident to $U$ go to sets $V_i$ with $|N^*_2(u)\cap V_i| \ge \eps |V_i|$. Moreover, since $u \in X_0$, there are (with high probability) at most 
$$p|U| \cdot \bigg( \frac{1}{2} + 2d \bigg) |V_i|  + dp^2 n \cdot |V_i| \, \le \, \bigg( \frac{1}{6} + 2d \bigg) p^2 n |V_i|$$
edges from $V_i$ to $U$ for every $i \in [k]$. It follows that there must be at least 
$$\bigg( \frac{2}{3} - 9d \bigg) k$$
indices $i \in [k]$ such that $|N^*_2(u)\cap V_i| \ge \eps |V_i|$. Since $\delta(R) \ge \big( \frac{1}{3} + \gamma \big) k$, it follows that there is an edge of $R$ between two such indices, as required.
\end{claimproof}

Now, suppose (for a contradiction) that there exists a vertex $u \in X_0$, and let $(V_i,V_j)$ be the pair given by the claim. By the definition of $(\eps,d,p)$-regularity, there exist vertices $v \in N^*_2(u)\cap V_i$ and $w \in N^*_2(u)\cap V_j$ such that $vw$ is an edge of $G$, and (by the definition of the robust second neighbourhood) there exist distinct vertices $v'$ and $w'$ in the common neighbourhoods of $u$ and $v$, and of $u$ and $w$, respectively. But then $uv'vww'$ forms a copy of $C_5$ in $G$, and so we have the desired contradiction. 
\end{proof}

Finally, let us prove Theorem~\ref{thm:Clong}. It follows from Proposition~\ref{prop:verysmallp} that $\delta_\chi(C_{2k+1},p) = 1$ whenever $\tfrac{\log n}{n} \ll p \ll n^{-(2k-1)/2k}$, and from Theorem~\ref{thm:CGS} and Proposition~\ref{prop:C5:simon} that $\delta_\chi(C_{2k+1},p) = 1/2$ for every function $n^{-(2k-1)/2k} \ll p \ll n^{-(2k-3)/(2k-2)}$. Thus, to deduce the theorem it will suffice to prove the following proposition.

\begin{prop}\label{prop:longcycles}
$\delta_\chi(C_{2\ell+1},p) = 0$ for every $\ell \ge 3$ and $n^{-1/2} \ll p = o(1)$.
\end{prop}

\begin{proof}
Let $\gamma > 0$, and let $G$ be a $C_{2\ell+1}$-free spanning subgraph of $G(n,p)$ with
$$\delta(G) \ge 2\gamma pn.$$
Applying the sparse minimum degree form of Szemer\'edi's regularity lemma to $G$, we obtain a partition $V(G) = V_0 \cup V_1 \cup \cdots \cup V_k$, such that the reduced graph $R$ satisfies $\delta(R) \ge  \gamma k$. Define
\begin{equation}\label{def:Xi:NR}
X_i := \Big\{ v \in V(G) : \big| N^*_2(v) \cap V_i \big| \ge d |V_i| \Big\}
\end{equation}
for each $i \in [k]$, where $N^*_2(v)$ is as defined in the proof of Proposition~\ref{prop:C5:thirdupper}.
Set $X_0:=V(G)\setminus\big(X_1\cup\dots\cup X_k\big)$. We will again show that $X_i$ is an independent set for each $1 \le i \le k$, and that $X_0 = \emptyset$. 

This time it is relatively easy to see that $X_0 = \emptyset$, since if $v \in X_0$ then $|N^*_2(v)| \le 2dn$, which implies that there are (with high probability) at most 
$$2p \cdot |U| \cdot 2dn + dp^2 n \cdot n \, = \, \big( 8\gamma d + d \big) p^2 n^2 \, < \, \gamma^2 p^2 n^2$$
edges incident to a set $U \subset N(v)$ with $|U| = 2\gamma pn$, contradicting our assumption on the minimum degree of $G$. 

Therefore, let us suppose that $uv \in E(G)$ for some $u,v \in X_i$, and observe that $N_2^*(u) \cap V_i$ and $N_2^*(v) \cap V_i$ both have size at least $d |V_i|$. Let $V_j$ be such that $(V_i,V_j)$ is an $(\eps,d,p)$-regular pair in $G$ (this exists because $\delta(R) > 0$). We claim that  there exists a path of length $2\ell-4$ in $G[V_i,V_j]$ from $N_2^*(u) \cap V_i$ to $N_2^*(v) \cap V_i$ which uses neither $u$ nor $v$. 

To find the desired path, we first construct a sequence of sets $(B_1,\ldots, B_{2\ell-3})$ as follows. Set $B_1 = N_2^*(u)$, and define $B_t$ (for each $2 \le t \le 2\ell - 3$) to be the set of vertices of $G$ with at least $2\ell$ neighbours in $B_{t-1}$. Observe that
$$|B_{2t+1} \cap V_i| \ge (1 - \eps) |V_i| \qquad \textup{and} \qquad |B_{2t} \cap V_j| \ge (1 - \eps) |V_j|$$
for every $1 \le t \le \ell - 2$, since the first time this fails we obtain a contradiction to the definition of an $(\eps,d,p)$-regular pair. Thus 
$$|B_{2\ell-3} \cap N_2^*(v)| \, \ge \, |N_2^*(v) \cap V_i| - \eps |V_i| \, \ge \, (d - \eps) |V_i| \, > \, 2\ell.$$ 
We can now choose the vertices $(w_1,\ldots,w_{2\ell-3})$ of the path greedily, one by one, with $w_t \in B_t$, avoiding all previously-chosen vertices and $u$ and $v$. Finally we choose vertices $u'$ and $v'$ (not so far used) in the common neighbourhoods of $u$ and $w_1$ and $v$ and $w_{2\ell-3}$ respectively, to complete a $(2\ell+1)$-vertex cycle. But we assumed that the graph $G$ is $C_{2\ell+1}$-free, so this contradiction completes the proof of the proposition.
\end{proof}

\section{Open problems}
\label{sec:open}

A large number of interesting questions about $\delta_\chi(H,p)$ remain wide open, and we hope that the work in this paper will inspire further research on the topic. In this section we will mention just a few of (what seem to us) the most natural open problems. We begin with the following (somewhat tentative) conjecture, which says that the hypothesis of Theorem~\ref{thm:classhigh} cannot be weakened to $\chi(H) \ne 3$. 

\begin{conj}\label{conj:chi4}
There exists a graph $H$ with $\chi(H) = 4$ and a function $p = p (n)$ with
$$n^{-1/m_2(H)} \ll p \ll n^{-1/2}$$
such that $\delta_\chi(H,p) < \pi(H)$. 
\end{conj}

We remark that, as far as we know, this might even be true for \emph{every} graph $H$ with $\chi(H) = 4$ and $m_2(H) < 2$ and every function $p$ in this range. It would therefore also be interesting to give a counter-example to this stronger statement. 

In the case $\chi(H) = 3$, we have barely begun to understand the behaviour of $\delta_\chi(H,p)$. An important next step would be to resolve the problem for all odd cycles.

\begin{prob}
Determine $\delta_\chi(C_{2k+1},p)$ for $k \ge 3$ in the range $n^{-(2k-3)/(2k-2)} \ll p \ll n^{-1/2}$.
\end{prob}

We suspect that $\delta_\chi(C_{2k+1},p) = 0$ for a much wider range of $p$ than those we considered in Theorem~\ref{thm:Clong}; in particular, it seems plausible that by considering robust $t$\textsuperscript{th} neighbourhoods for $2 \le t \le k-1$ (rather than just second neighbourhoods) the proof of Proposition~\ref{prop:longcycles} could be modified to prove this for all $p \gg n^{-(k-2)/(k-1)}$. Nevertheless, a significant gap remains, and we are not sure what to expect in the range $n^{-(2k-3)/(2k-2)} \ll p \ll n^{-(k-2)/(k-1)}$. 

For more general 3-chromatic graphs, it follows from Theorem~\ref{thm:CGS} and Proposition~\ref{prop:cloud:lower} that if $H$ is not a cloud-forest graph and $p = p(n)$ satisfies
\begin{equation}\label{eq:boundsonpforclouds}
\max\big\{ n^{-1/m_2(H)}, n^{-1/2} \big\} \, \ll \, p \, \ll \, 1,
\end{equation}
then $\delta_{\chi}(H,p) = 1/2$. For cloud-forest graphs that are not thundercloud-forest we can at present only determine $\delta_{\chi}(H,p)$ in a smaller range: it follows from Proposition~\ref{prop:thundercloud:lower} and~\cite[Proposition~4.1]{dense} that $\delta_{\chi}(H,p) = 1/3$ if $p = o(1)$ and $p = n^{-o(1)}$. In terms of lower bounds, we have $\delta_\chi(H,p)\ge 1/3$ in the range~\eqref{eq:boundsonpforclouds} by Proposition~\ref{prop:thundercloud:lower}, and $\delta_\chi(H,p)=1$ when $p\ll n^{-1/m_2(H)}$ by Proposition~\ref{prop:verysmallp}; if $m_2(H)>2$ then these ranges overlap. In this case it is possible that these lower bounds are correct, though we cannot prove it. When $m_2(H)\le 2$ we do not know what to expect in the gap between the two ranges.

\begin{qu}
Let $H$ be a cloud-forest graph, and suppose that $p = p(n)$ satisfies~\eqref{eq:boundsonpforclouds}. 
Is $\delta_\chi(H,p) = 1/3$ whenever $H$ is a not a thundercloud-forest graph?
\end{qu}

Recall that for thundercloud-forest graphs, we do not even know how to determine $\delta_\chi(H,p)$ in the range $p = n^{-o(1)}$, see~\cite[Conjecture~1.6]{dense}.

Finally, we do not know how to determine the behaviour of $\delta_\chi(H,p)$ around the threshold $p = n^{-1/m_2(H)}$.

\begin{prob}
Determine $\delta_\chi(H,p)$ for $p = c n^{-1/m_2(H)}$. 
\end{prob}

The construction used to prove Proposition~\ref{prop:verysmallp} (that is, randomly remove one edge from each copy of $H$ in $G(n,p)$) can easily be extended to give a lower bound in the range $(\pi(H),1)$ for all sufficiently small $c > 0$. It would be interesting to understand how the extremal graph transitions (as $c$ increases) from a random-like graph into a Tur\'an-like graph.

%%%% BIBLIOGRAPHY %%%%%%%%%%%%%%%%%%%%%%%%%%%%%%%%%%%%%%%%%%%%%%%%%%%%%

\bibliographystyle{amsplain_yk} 
\bibliography{ChromaticThreshold}

\end{document}